\documentclass[draft]{amsproc}
\usepackage{amsmath}
\usepackage{amsfonts}
\usepackage[dvips]{graphicx}
\usepackage{amssymb}
\usepackage{amsbsy}
\usepackage{amsthm}
\usepackage{graphics}
\usepackage{color}
\usepackage{textcomp}

\makeatletter
\@namedef{subjclassname@2010}{%
\textup{2010} Mathematics Subject Classification}
\makeatother

\numberwithin{equation}{section}

\newtheorem{theorem}{Theorem}[section]
\newtheorem{corollary}[theorem]{Corollary}
\newtheorem{prop}[theorem]{Proposition}

\newtheorem{lemma}[theorem]{Lemma}
\newtheorem{case}{Case}

\newcommand{\ttt}{\mathcal{T}}

\newcommand{\aaa}{\mathbb{A}}
\newcommand{\ddd}{\mathcal{D}}

\theoremstyle{definition}
\newtheorem{ex}[theorem]{Example}

\newcommand*{\Le}{\leqslant}
\newcommand*{\Ge}{\geqslant}

\begin{document}

\title[The single equality $A^{*n}A^n = (A^*A)^n$ does  not
imply quasinormality]{The single equality $A^{*n}A^n =
(A^*A)^n$ does not imply the quasinormality of
weighted shifts on rootless directed trees}
   \author[P. Pietrzycki]{Pawe{\l} Pietrzycki}
   \subjclass[2010]{Primary 47B20, 47B33; Secondary
47B37} \keywords{Quasinormal operator, bilateral
weighted shift, weighted shift on a directed tree,
composition operator in an $L^2$-space}
   \address{Wydzia{\l} Matematyki i Informatyki, Uniwersytet
Jagiello\'{n}ski, ul. {\L}ojasiewicza 6, PL-30348
Krak\'{o}w}
   \email{pawel.pietrzycki@im.uj.edu.pl}
   \begin{abstract}
It is proved that each bounded injective bilateral
weighted shift $W$ satisfying the equality
$W^{*n}W^{n}=(W^{*}W)^{n}$ for some integer $n\Ge 2$
is quasinormal. For any integer $n\Ge 2$, an example
of a bounded non-quasinormal weighted shift $A$ on a
rootless directed tree with one branching vertex which
satisfies the equality $A^{*n}A^{n}=(A^{*}A)^{n}$ is
constructed. It is also shown that such an example can
be constructed in the class of composition operators
in $L^2$-spaces over $\sigma$-finite measure spaces.
   \end{abstract}
   \maketitle
   \section{Introduction}
The class of bounded quasinormal operators was
introduced by A. Brown in \cite{brow}. Two different
definitions of unbounded quasinormal operators
appeared independently in \cite{kauf} and in
\cite{szaf}. As recently shown in \cite{jabl}, these
two definitions are equivalent. Following \cite{szaf},
we say that a closed densely defined operator $A$ in a
complex Hilbert space $\mathcal{H}$ is {\em
quasinormal} if $A$ commutes with the spectral measure
$E$ of $|A|$, i.e $E(\sigma) A \subset CE(\sigma)$ for
all Borel subsets $\sigma$ of the nonnegative part of
the real line. By \cite[Proposition 1]{szaf}, a closed
densely defined operator $A$ in $\mathcal{H}$ is
quasinormal if and only if $U|A|\subset |A|U$, where
$A = U|A|$ is the polar decomposition of $A$ (cf.\
\cite[Theorem 7.20]{weid}). It is well-known that
quasinormal operators are always subnormal and that
the reverse implication does not hold in general. Yet
another characterization of quasinormality of
unbounded operators states that a closed densely
defined operator $A$ is quasinormal if and only if the
equality $A^{*n}A^{n}=(A^{*}A)^{n}$ holds for $n=2,3$
(see \cite[Theorem 3.6]{jabl}; see also \cite{Jib} for
the case of bounded operators and \cite[p.\ 63]{embry}
for a prototype of this characterization). For more
information on quasinormal operators we refer the
reader to \cite{brow,conw}, the bounded case, and to
\cite{kauf,szaf,maj,jabl}, the unbounded one.

In view of the above discussion, the question arises
as to whether the single equality
$A^{*n}A^{n}=(A^{*}A)^{n}$ with $n\Ge 2$ implies the
quasinormality of $A$. It turns out that the answer to
this question is in the negative. In fact, as recently
shown in \cite[Example 5.5]{jabl}, for every integer
$n\Ge 2$, there exists a weighted shift $A$ on a
rooted and leafless directed tree with one branching
vertex such that
   \begin{equation}\label{qqq}
\text{$(A^{*} A)^n=A^{*n}A^{n}$ and $(A^{*} A)^k \neq
A^{*k}A^{k}$ for all $k\in\{2,3,\ldots\}\setminus
\{n\}$.}
   \end{equation}
It remained an open question as to whether such
construction is possible on a rootless and leafless
directed tree. This is strongly related to the
question of the existence of a composition operator
$A$ in an $L^2$-space (over a $\sigma$-finite measure
space) which satisfies \eqref{qqq}. In this paper, we
will construct for every integer $n\Ge 2$ examples of
bounded (necessarily non-quasinormal) weighted shifts
$A$ on a rootless and leafless directed tree with one
branching vertex which satisfy \eqref{qqq} (cf.\
Theorem \ref{glowne}). This combined with the fact
that every weighted shift on a rootless directed tree
with nonzero weights is unitarily equivalent to a
composition operator in an $L^2$-space (see
\cite[Theorem 3.2.1]{memo} and \cite[Lemma 4.3.1]{9})
yields examples of composition operators satisfying
\eqref{qqq} (cf.\ Theorem \ref{glowne}).

It was observed in \cite[p.\ 144]{jabl} that a
unilateral or a bilateral injective weighted shift $W$
which satisfies the equality $(W^{*}W)^k=(W^{*})^k
W^k$ for $k=2$ is quasinormal, and that the same is
true for $k=3$ provided $W$ is bounded. In the present
paper we will show that in the class of bounded
injective bilateral weighted shifts, the single
equality $W^{*n}W^{n}=(W^{*}W)^{n}$ with $n\Ge 2$ does
imply quasinormality (cf.\ Theorem \ref{mainX}). This
is no longer true for unbounded ones even for $k=3$
(cf.\ Example \ref{przyk}).
   \section{Preleminaries}
In this paper we use the following notation. The
fields of rational, algebraic, real and complex
numbers are denoted by $\mathbb{Q}$, $\mathbb{A}$,
$\mathbb{R}$ and $\mathbb{C}$, respectively. The
symbols $\mathbb{Z}$, $\mathbb{Z}_{+}$, $\mathbb{N}$
and $\mathbb{R}_+$ stand for the sets of integers,
nonnegative integers, positive integers and
nonnegative real numbers, respectively. The field of
rational functions in $x$ with rational coefficients
is denoted by $\mathbb{Q}(x)$. We write $\mathbb Z[x]$
for the ring of all polynomials in $x$ with integer
coefficients.

Let $A$ be a linear operator in a complex Hilbert
space $\mathcal{H}$. Denote by $\mathcal{D}(A)$, $\bar
A$ and $A^*$ the domain, the closure and the adjoint
of $A$ respectively (provided they exist). A subspace
$\mathcal E$ of $\mathcal{D}(A)$ is called a {\em
core} for $A$ if $\mathcal E$ is dense in
$\mathcal{D}(A)$ with respect to the graph norm of
$A$. We write $\boldsymbol{B}(\mathcal{H})$ for the
set of all bounded operators in $\mathcal{H}$ whose
domain are equal to $\mathcal{H}$.

In the present paper, by a {\em classical weighted
shift} we mean either a unilateral weighted shift $W$
in $\ell^2$ or a bilateral weighted shift $W$ in
$\ell^2(\mathbb{Z})$. To be more precise, $W$ is
understood as the product $VD$, where, in the
unilateral case, $V$ is the unilateral isometric shift
on $\ell^2$ of multiplicity 1 and $D$ is a diagonal
operator in $\ell^2$ with diagonal elements
$\{\lambda_n\}_{n=0}^{\infty}$; in the bilateral case,
$V$ is the bilateral unitary shift on $\ell^2
(\mathbb{Z})$ of multiplicity 1 and $D$ is a diagonal
operator in $\ell^2 (\mathbb{Z})$ with diagonal
elements $\{\lambda_n\}_{n=-\infty}^{\infty}$. In
fact, $W$ is a unique closed linear operator in
$\ell^2$ (respectively, $\ell^2 (\mathbb{Z})$) such
that the linear span of the standard orthonormal basis
$\{e_n\}_{n=0}^{\infty}$ of $\ell^2$ (respectively,
$\{e_n\}_{n=-\infty}^{\infty}$ of $\ell^2
(\mathbb{Z})$) is a core for $W$ and
   \begin{equation}  \label{Wen}
We_n = \lambda_n e_{n+1} \quad \textup{for}\quad n\in
\mathbb{Z}_{+} \quad(\textup{respectively, } n\in
\mathbb{Z}).
   \end{equation}
Suppose $\ttt = (V; E)$ is a directed tree ($V$ and
$E$ are the sets of vertices and edges of $\ttt$,
respectively). If $\ttt$ has a root, we denote it by
$\textrm{ root}$. Put $V^{\circ}= V
\setminus\{\textrm{root}\}$ if $\ttt$ has a root and
$V^{\circ}=V$ otherwise. For every $u \in V^{\circ}$,
there exists a unique $v \in V$, denoted by $par(u)$,
such that $(v; u) \in E$. For any vertex $u \in V$ we
put $Chi(u) = \{v \in V : (u, v) \in E\}$. The Hilbert
space of square summable complex functions on $V$
equipped with the standard inner product is denoted by
$\ell^2(V )$. For $u \in V$, we define $e_u \in
\ell^2(V)$ to be the characteristic function of the
one-point set $\{u\}$.

Given a system $ \lambda=\{\lambda_v \}_{v \in
V^{\circ}} $ of complex numbers, we define the
operator $S_\lambda$ in $\ell^2(V)$, which is called a
\textit{weighted shift} on $\ttt$ with weights
$\lambda$, as follows
   \begin{equation*}
\ddd(S_{\lambda}) =\{ f\in \ell^2(V ):
\varLambda_{\ttt} f \in \ell^2(V )\}\quad \textup{and}
\quad S_{\lambda} = \varLambda_{\ttt} f \quad
\textup{for} \quad f \in \ddd(S_\lambda),
   \end{equation*}
where
   \begin{displaymath}
(\varLambda_{\ttt} f)(v) =\left\{\begin{array}{ll}
\lambda_v f(par(v)) & \textrm{if  $v \in V^{\circ},$}\\
0 & \textrm{otherwise}.\end{array} \right.
   \end{displaymath}
We refer the reader to \cite{memo} for more details on
weighted shifts on directed trees and their relations
to classical weighted shifts.

Let us recall some useful properties of weighted
shifts on directed trees we need in this paper.
   \begin{prop}[\mbox{\cite[Proposition 3.1.3]{memo}}]
Let $S_\lambda$ be a weighted shift on a directed tree
$\ttt$ with weights $ \lambda=\{\lambda_v \}_{v \in
V^{\circ}}$. Then the following assertions hold:
   \begin{enumerate}
\item[(i)] $S_\lambda$ is a closed operator,
\item[(ii)] $e_u \in \ddd(S_\lambda)$ if and only if $\sum_{v\in Chi(u)}
|\lambda_v|^2<\infty$ and in this case
   \begin{equation*}
S_\lambda e_u=\sum_{v\in Chi(u)} \lambda_v e_v, \qquad
\|S_\lambda e_u\|=\sum_{v\in Chi(u)} |\lambda_v|^2,
   \end{equation*}
   \item[(iii)] $S_\lambda$ is densely defined if
and only if $e_u \in D(S_\lambda) $ for every $u\in
V$.
   \end{enumerate}
   \end{prop}
   \begin{prop}[\mbox{\cite[Proposition 3.1.8]{memo}}]
   \label{ogrS} Let $S_\lambda$ be a weighted shift on
a directed tree $\ttt$ with weights $\lambda =
\{\lambda_u\}_{u\in V^\circ}$. Then the following
conditions are equivalent:
\begin{enumerate}
\item[(i)] $\ddd(S_\lambda)=\ell^2(V)$,

\item[(ii)] $S_\lambda \in  \boldsymbol B(\ell^2(V ))$,

\item[(iii)] $\sup_{u\in V} \sum_{v\in Chi(u)}
|\lambda_v|^2 < \infty$.
\end{enumerate}
Moreover, if $ S_\lambda \in \boldsymbol B(\ell^2(V
))$, then
   \begin{equation*}
\|S_\lambda\| = \sup_{u\in V} \|S_\lambda e_u\| =
\sqrt{\sup_{u\in V}\sum_{v\in Chi(u)} |\lambda_v|^2}.
   \end{equation*}
\end{prop}
\begin{prop}[\mbox{\cite[Proposition 8.1.7]{memo}}]
\label{js} Let $n \in \mathbb{Z_+}$. If $S_\lambda\in
\boldsymbol B(\ell^2(V ))$ is a weighted shift on a
directed tree $\ttt$ with weights $\lambda =
\{\lambda_v\}_{v\in V^\circ}$, then the following two
conditions are equivalent:
\end{prop}
\begin{enumerate}
\item[(i)] $(S^{*}_{\lambda}S_{\lambda})^n=(S^{*}_{\lambda})^nS^{n}_{\lambda}$,
\item[(ii)] $\parallel S_{\lambda}e_u\parallel^n=\parallel S^n_{\lambda}e_u
\parallel $ for all $u\in V$.
\end{enumerate}

   The basic facts on bounded composition operators in
$L^2$-spaces we need in this paper can be found in
\cite{nor} (see also \cite{b-j-j-s} for the case of
unbounded composition operators).
   \section{ Transcendentality of $\ln(\alpha)$}
The irrationality of $e$ was established by Euler in
1744 and that of $\pi$ was proven by Johann Heinrich
Lambert in 1761. Their transcendence was proved about
a century later by Hermite and Lindemann respectively.
A generalisation of the above result was given by
Weierstrass in 1885, and is as follows.
   \begin{theorem}\label{wei}\cite[Theorem 1.4]{przest}
$($Lindemann-Weierstrass theorem$)$. For any finite
system of distinct algebraic numbers
$\alpha_1,\ldots,\alpha_{n}$, the numbers
$e^{\alpha_1}, \ldots, e^{\alpha_{n}}$ are linearly
independent over $\aaa$.
   \end{theorem}
For the reader's convenience, we include the proof of
the following result which is surely folklore. This
fact will be used in Section \ref{Sec5}.
   \begin{corollary} \label{logarytm} $\ln(\alpha)$ is
transcendental for any algebraic number $\alpha\neq
0,1$.
\end{corollary}
   \begin{proof} Suppose that, contrary to our claim,
$\ln(\alpha)$ is algebraic. Then, by Theorem \ref{wei}
with $\alpha _1 = 0$ and $\alpha_2 = \ln(\alpha)$, we
see that $1$ and $\alpha=e^{\ln(\alpha)}$ are linearly
independent over $\aaa$, which gives a contradiction.
   \end{proof}
   \section{Bounded classical weighted shifts}
In this section, we will show that for every integer
$n$ greater than or equal to $2$, any bounded
bilateral weighted shift $W$ satisfying the equation
$(W^{*}W)^n=(W^{*})^n W^n$ is quasinormal (to simplify
terminology, we drop the adjective ``classical'' in
this section). We begin by proving two key lemmata.
   \begin{lemma}  \label{charp}
Let $k\in \mathbb{N}$. Then each root of the
polynomial
   \begin{align} \label{chara}
p(z)=kz^k-(z^{k-1}+z^{k-2}+ \ldots +1), \quad z \in
\mathbb C,
   \end{align}
except for $z=1$, is in the open unit disk centered at
$0$. Moreover, all roots of the polynomial $p(z)$ are
of multiplicity one.
   \end{lemma}
   \begin{proof}
It is enough to consider the case of $k \Ge 2$.
Suppose $z$ is a complex number such $|z|>1$ and
$p(z)=0$. Then $|z^k|>|z^i|$ for $i=0,1,\ldots,k-1$.
This implies that
   \begin{equation*}
|kz^k|>|z^{k-1}+z^{k-2}+\ldots+1|=|kz^k|,
   \end{equation*}
which gives a contradiction. This means that all roots
of the polynomial $p(z)$ satisfies the inequality
$|z|\Le 1$.

It is clear that $p(1)=0$. We show that $1$ is in fact
the only root which lies on the unit circle. Suppose
that there exists $z\in \mathbb{C}$ such that $|z|=1$,
$p(z)=0$ and $z\neq 1$. Since the polynomial $p(z)$
has real coefficients, $\bar{z}=\frac{1}{z}$ is its
root as well. Hence we have
   \begin{align}\label{tu}
0&=kz^k-(z^{k-1}+\ldots+1),
   \\ \label{tu2}
0&=k\Big(\frac{1}{z}\Big)^k-\Big(\Big(\frac{1}{z}
\Big)^{k-1}+\ldots+1\Big).
   \end{align}
It follows from \eqref{tu2} that
   \begin{equation*}
k \frac{1}{z}=1+z+\ldots+z^{k-1}.
   \end{equation*}
This and (\ref{tu}) yield
   \begin{equation} \label{r2}
z^{k+1}=1.
   \end{equation}
On the other hand multiplying both sides of the
equality in \eqref{tu} by $z-1$ we get
   \begin{equation}\label{jjj}
kz^{k+1}-(k+1)z^k+1=0.
   \end{equation}
Applying \eqref{r2} and \eqref{jjj}, we see that
$z^k=1$ and so, by \eqref{r2}, $z=1$. This contradicts
the assumption that $z\neq 1$.

Now we will prove the ``moreover'' part of the
theorem. Using \cite[Theorem III.6.10]{Hug}, we easily
verify that $1$ is not a multiple root of the
polynomial $p(z)$. Suppose that the polynomial $p(z)$
has a multiple root different from $1$. Then clearly
the polynomial $q(z)=(z-1)p(z)$ has a multiple root
different from $1$. Applying \cite[Theorem
III.6.10]{Hug} again, we deduce that the polynomials
$q(z)$ and $q'(z)$ has a common root different from
$1$. Since the polynomials $q(z)= kz^{k+1}-(k+1)z^k+1$
and $q'(z)=k(k+1)z^{k-1}(z-1)$ have only one common
root $1$, we get a contradiction.
   \end{proof}
   \begin{lemma}\label{ociag} Let $k\in \mathbb
N$. Suppose that $\{a_n\}^{\infty}_{n=-\infty}$ is a
bounded sequence of real numbers that satisfies the
following recurrence relation
   \begin{equation*}
ka_n=a_{n+1}+\ldots+a_{n+k}, \quad n \in \mathbb Z.
   \end{equation*}
Then $\{a_n\}^{\infty}_{n=-\infty}$ is a constant
sequence.
   \end{lemma}
   \begin{proof}
Without loss of generality, we can assume that $k \Ge
2$. Suppose that, contrary to our claim, the sequence
$\{a_n\}^{\infty}_{n=-\infty}$ is not constant. Then
there exist $r\in \mathbb{Z}$ and $\varepsilon>0$ such
that
   \begin{eqnarray}\label{eps}
|a_r - a_{r-1}| > \varepsilon.
   \end{eqnarray}
Given $l\in \mathbb{N}$, we define the sequence
$\{b^{(l)}_n\}_{n=0}^\infty$ by
   \begin{align}  \label{recrel}
b^{(l)}_n = a_{l-n}, \quad n \in \mathbb{Z}_+.
   \end{align}
Clearly, $\{b^{(l)}_n\}_{n=0}^\infty$ satisfies the
recurrence relation
   \begin{equation} \label{chara2}
kb^{(l)}_{n+k}=b^{(l)}_n+\ldots+b^{(l)}_{n+k-1}, \quad
n \in \mathbb Z_+,
   \end{equation}
with the initial values $b^{(l)}_0=a_{l}$, \ldots,
$b^{(l)}_{k-1}=a_{l-k+1}$. The polynomial
$\frac{1}{k}p(z)$, where $p(z)$ is as in
\eqref{chara}, is the characteristic polynomial of the
recurrence relation \eqref{chara2}. By Lemma
\ref{charp} and \cite[Theorem 3.1.1]{hall}, we have
   \begin{align}\label{rr}
b^{(l)}_n=A_1^{(l)}z_1^n + \ldots +
A_{k}^{(l)}z_{k}^n, \quad n \in \mathbb Z_+,
   \end{align}
where $z_1, \ldots, z_{k}$ are the roots of the
polynomial $\frac{1}{k}p(z)$, and $A^{(l)}_1, \ldots,
A^{(l)}_{k}$ are complex numbers depending on the
initial values $b^{(l)}_0, \ldots, b^{(l)}_{k-1}$. In
view of Lemma \ref{charp}, we may assume that $z_1=1$.
It follows from \eqref{rr} that $A^{(l)}_1, \ldots,
A^{(l)}_{k}$ is a solution of the system of linear
equations in unknowns $w_1, \ldots, w_k$:
   \begin{equation}\label{qwerty}
   \left\{
   \begin{array}{l}
b^{(l)}_0=w_1+\ldots+w_{k}\\
b^{(l)}_1=w_1z_1+\ldots+w_{k}z_{k}
   \\
   \vdots
   \\
b^{(l)}_{k-1}=w_1z_1^{k-1}+\ldots+w_{k}z_{k}^{k-1}.
   \end{array}
   \right.
   \end{equation}
Let $U$ be the matrix associated with the system
\eqref{qwerty}, i.e.,
   \begin{displaymath}
U =\left [\begin{array}{llll} 1 & 1 &\cdots& 1
   \\
z_1 & z_2 &\cdots&z_k
   \\
\vdots & \vdots&\cdots&\vdots
   \\
z_1^{k-1}&z_2^{k-1}&\cdots &
z_{k}^{k-1}\end{array}\right].
   \end{displaymath}
Since $U$ is a Vandermonde matrix, we deduce that
$\det U \neq 0$. Hence the system \eqref{qwerty} has a
unique solution which, by Cramer's Rule (cf.\
\cite[Corollary VII.3.8]{Hug}), is given by
   \begin{equation}   \label{alf}
A^{(l)}_j=\frac{\det U^{(l)}_j}{\det U}, \quad j = 1,
\ldots, k,
   \end{equation}
where $U_j^{(l)}$ is the matrix formed by replacing
the $j$th column of $U$ by the transpose of the row
vector $[b^{(l)}_0,b^{(l)}_1,
\ldots,b^{(l)}_{k-1}]=[a_l, a_{l-1}, \ldots
a_{l-k+1}]$.

By assumption, $C:=\sup\{|a_n|\colon n \in \mathbb Z\}
< \infty$. Set $L=k!C$. Note that
   \begin{equation} \label{dete}
|\det U_j^{(l)}| \Le L, \quad j=1, 2, \ldots, k.
   \end{equation}
Indeed, this can be deduced by estimating each summand
of $\det U^{(l)}_j$ (cf.\ \cite[Theorem VII.3.5]{Hug})
and using the fact that $z_1=1$ and
$p:=\max\{|z_i|\colon i=2, \ldots, k\} < 1$ (cf.\
Lemma \ref{charp}). Take $l\in \mathbb N$ such that $l
\Ge r$ and $2(k-1) p^{l-r} \frac{L}{|\det U|}\notag
<\varepsilon$. This combined with \eqref{alf} and
\eqref{dete} yields
   \begin{align*}
|a_r - a_{r-1}| & \overset{\eqref{recrel}}=
|b^{(l)}_{l-r}-b^{(l)}_{l-r+1}|
   \\ &\overset{\eqref{rr}}=|A_1^{(l)}z_1^{l-r} + \ldots
+ A_k^{(l)} z_{k}^{l-r} - (A_1^{(l)} z_1^{l-r+1} +
\ldots+A_k^{(l)}z_{k}^{l-r+1})|
   \\
& \hspace{+1.5ex} =|A_2^{(l)}z_2^{l-r} + \ldots +
A_k^{(l)} z_{k}^{l-r} - (A_2^{(l)} z_2^{l-r+1} +
\ldots+A_k^{(l)}z_{k}^{l-r+1})|
   \\
&\hspace{+1.5ex} \Le 2(k-1) p^{l-r} \frac{L}{|\det U|}
< \varepsilon,
   \end{align*}
which contradicts (\ref{eps}). This completes the
proof.
   \end{proof}
Now we are ready to prove the main result of this
section. Recall the well-known and easy to prove fact
that a quasinormal injective bilateral weighted shift
is a multiple of a unitary operator.
   \begin{theorem}  \label{mainX}
   Let $k\Ge 2$. Then any bounded
injective bilateral weighted shift $W$ that satisfies
the equality $(W^{*}W)^k=(W^{*})^k W^k$ is
quasinormal.
   \end{theorem}
   \begin{proof}  Let $W$ be a bounded injective bilateral weighted
shift with weights
$\{\lambda_n\}_{n=-\infty}^{\infty}$ (cf.\
\eqref{Wen}). Without loss of generality, we can
assume that $\lambda_n > 0$ for all $n\in \mathbb{Z}$
(cf.\ \cite{Shi}). Suppose that
$(W^{*}W)^k=(W^{*})^kW^k$ for some $k \Ge 2$. By
Proposition \ref{js}, we have
   \begin{eqnarray} \label{w11}
\lambda_n^k=\lambda_n
\lambda_{n+1}\cdots\lambda_{n+k-1}, \quad n \in
\mathbb{Z}.
   \end{eqnarray}
Since that operator $W$ is bounded, the sequence
$\{\lambda_n\}^{\infty}_{n=-\infty}$ is bounded as
well. We will show that there exists $c\in (0,\infty)$
such that $\lambda_n>c$ for every $n\in \mathbb{N}$.
If not, there exists a subsequence
$\{\lambda_{n_i}\}_{i=1}^{\infty}$ of
$\{\lambda_{n}\}_{n=1}^{\infty}$ such that
$\lambda_{n_i}\rightarrow 0$ as $i\rightarrow \infty$.
Set
   \begin{align}\label{dd}
d=\min_{i=1,\ldots,k}\lambda_i
   \end{align}
and
   $$D=\sup_{i\in \mathbb{Z}}\lambda_i.
   $$
Then there exists $m\in \mathbb{N}$ such that $n_m >
k$ and $\lambda_{n_m}<\frac{d^k}{D^{k-1}}$. By
\eqref{w11} we have
   \begin{displaymath}
\lambda_{n_m-i}^k=\lambda_{n_m-i}
\lambda_{n_m-i+1}\cdots\lambda_{n_m-i+k-1}
<\frac{d^k}{D^{k-1}}D^{k-1}=d^k, \quad
i=0,1,\ldots,k-1.
   \end{displaymath}
   Hence
   \begin{equation}\label{ww}
\lambda_{n_m-i}<d, \quad i=0,1, \ldots ,k-1.
   \end{equation}
Since each term of the sequence
$\{\lambda_i\}^{n_m-k}_{i=-\infty}$ is a geometric
mean of $k-1$ positive real numbers smaller then $d$,
we deduce from \eqref{w11} and \eqref{ww} that
$\lambda_n<d$ for every $n\Le n_m$. In particular,
$\lambda_i<d$ for $i=1,2,\ldots,k$, which contradicts
(\ref{dd}). Applying \eqref{w11} again, we easily see
that the sequence $\{\lambda_i\}^0_{i=-\infty}$ is
bounded below from zero. Altogether this implies that
the whole sequence
$\{\lambda_i\}^{\infty}_{i=-\infty}$ is bounded below
from zero.

Now we define the sequence
$\{a_n\}^{\infty}_{n=-\infty}$ by $a_n=\log
\lambda_n$. In view of the previous paragraph, the
sequence $\{a_n\}^{\infty}_{n=-\infty}$ is bounded. It
follows from \eqref{w11} that
$\{a_n\}^{\infty}_{n=-\infty}$ satisfies the following
recurrence relation
   \begin{displaymath}
(k-1)a_n=a_{n+1} + \ldots + a_{n+k-1}, \quad n \in
\mathbb{Z}.
   \end{displaymath}
Hence, by Lemma \ref{ociag}, the sequence
$\{a_n\}^{\infty}_{n=-\infty}$ is constant. It is
easily seen that $W$ is a multiple of a unitary
operator and as such is quasinormal. This completes
the proof.
   \end{proof}
It is worth pointing out that Theorem \ref{mainX} is
no longer true if the bilateral weighted shift is not
bounded.
   \begin{ex}   \label{przyk}
Let $W$ be an injective bilateral weighted shift with
weights $\{\lambda_n\}_{n=-\infty}^{\infty}$ given by
$\lambda_n = \exp((-2)^n)$ for $n\in \mathbb{Z}$. Then
the sequence $\{\lambda_n\}_{n=-\infty}^{\infty}$
satisfies \eqref{w11} with $k=3$, but it does not
satisfy \eqref{w11} with $k=2$. Denote by
$\mathcal{E}$ the linear span of the standard
orthonormal basis $\{e_n\}_{n=-\infty}^{\infty}$ of
$\ell^2(\mathbb{Z})$. Clearly $\mathcal{E} \subset
\mathcal{D}(W) \cap \mathcal{D}(W^*)$, $W(\mathcal{E})
\subset \mathcal{E}$ and $W^*(\mathcal{E}) \subset
\mathcal{E}$. This, by the von Neumann theorem (cf.\
\cite[Theorem 5.3]{weid}), implies that $W^3$ is
closable. Hence, we have
   \begin{align*}
(W^*W)^3|_{\mathcal{E}} = W^{*3}W^{3}|_{\mathcal{E}}
\subset W^{*3}W^{3} \subset (W^{3})^*W^{3} \subset
(\overline{W^{3}})^*\overline{W^{3}}.
   \end{align*}
Since $\mathcal{E}$ is a core for the selfadjoint
operator $(W^*W)^3$ (cf.\ \cite[Proposition
3.4.3]{memo}) and
$(\overline{W^{3}})^*\overline{W^{3}}$ is selfadjoint
(cf.\ \cite[Theorem 5.39]{weid}), we deduce from the
maximality of selfadjoint operators that
   \begin{align*}
(W^*W)^3 = W^{*3}W^{3} =
(\overline{W^{3}})^*\overline{W^{3}}.
   \end{align*}
It is easily seen that $(W^{*} W)^2 \neq W^{*2}W^{2}$.
Hence, $W$ is not quasinormal (cf.\ \cite[Lemma
3.5]{jabl}).
   \end{ex}
   \section{Weighted shifts on directed trees and composition operators}
   \label{Sec5} Our aim in this section is to
construct for every integer $n\Ge 2$ an injective
non-quasinormal weighted shift $S_\lambda\in
\boldsymbol B(\ell^2(V_{\infty}))$ on a directed tree
$\ttt_{\infty}=(V_{\infty},E_{\infty})$ satisfying the
condition \eqref{qqq} with $A=S_\lambda$, where
$\ttt_{\infty}=(V_{\infty},E_{\infty})$ is the
rootless directed tree with one branching vertex
defined by
   \begin{align} \label{drzewo}
V_{\infty}&=\{-k:k\in\mathbb{Z}_{+}\}\sqcup\{(i,j):i,j\in\mathbb{N}\},
   \\ \label{drzewo2}
E_{\infty}&=\{(-k,-k+1):k\in
\mathbb{N}\}\sqcup\{(0,(i,1)):i \in \mathbb{N}\}
   \\ \notag
&\hspace{26.2ex}\sqcup\{((i,j),(i,j+1)):i,j \in
\mathbb{N}\}.
   \end{align}
(The symbol "$\sqcup$" detonates disjoint union of
sets.) The weights of $S_\lambda$ are defined with the
help of three sequences $\{\alpha_i\}_{i=1}^{\infty}$,
$\{\beta_i\}_{i=1}^{\infty}$ and
$\{\gamma_i\}_{i=0}^{\infty}$ of positive real numbers
as follows
   \begin{align} \label{wagi1}
\lambda_v =
   \begin{cases}
   \alpha_i & \text{if } v=(i,1), \, i\in\mathbb{N},
\\
   \beta_i & \text{if } v=(i,j), \, i\in\mathbb{N}, \,
   j\Ge 2,
\\
   \gamma_i & \text{if } v=-i, \, i\in \mathbb{Z}_{+}.
   \end{cases}
   \end{align}
It is a matter of routine to verify that Proposition
\ref{js} takes now the following form.
   \begin{prop}\label{js2} If  $n\Ge 2$ and $S_{\lambda}$
is a weighted shift on the directed tree
$\ttt_{\infty}$ with weights \eqref{wagi1}, then
$(S^{*}_{\lambda}S_{\lambda})^n=S^{*n}_{\lambda}
S^{n}_{\lambda}$ if and only if the following three
conditions hold{\em :}
   \begin{enumerate}
   \item[(i)] $\gamma_{k+n-1}^n=\gamma_{k+n-1}
\gamma_{k+n-2}\cdots\gamma_{k}$ for all $k\in
\mathbb{Z}_{+}$,
      \item[(ii)]$\gamma_{n-i-1}^n=\gamma_{n-i-1}\cdots
\gamma_{0}\sqrt{\sum^{\infty}_{k=1}\alpha_k^2
\beta_k^{2(i-1)}}$ for all $i\in
\{1,2,\ldots,{n-1}\}$,
   \item[(iii)]$\big(\sqrt{\sum^{\infty}_{k=1} \alpha_k^2}\,\big)^n=\sqrt{\sum^{\infty}_{k=1}
\alpha_k^2 \beta_k^{2(n-1)}}$.
\end{enumerate}
\end{prop}
   Our next goal is to consider a sequence $\{S_k\}_{k
\in \mathbb{Z}}$ of functions on $(0,1)$ that will
play an essential role in the proof of Theorem
\ref{glowne}. Given $k \in \mathbb{Z}$, we define a
function $S_k:(0,1)\rightarrow (0,\infty)$ by
   \begin{equation*}
S_k(x)=1^k+2^kx+3^kx^2+\ldots, \quad x \in (0,1).
   \end{equation*}
The following formulas are well-known in classical
analysis:
   \begin{align} \label{claan}
\text{$S_0(x)=\frac{1}{1-x}$ and
$S_{-1}(x)=-\frac{\ln(1-x)}{x}$ for $x\in(0,1)$.}
   \end{align}
Below we collect some properties of the functions
$\{S_k\colon k \in \mathbb{Z}_{+}\}$ .
   \begin{lemma} \label{mk}
There exists a $($unique\/$)$ sequence
$\{m_k\}_{k=0}^{\infty} \subset \mathbb{Z}[x]$ such
that
   \begin{enumerate}
   \item[(i)] $m_0=m_1=1$,
   \item[(ii)] the degree of $m_k$ is equal to $k-1$
for every $k\in \mathbb N$,
   \item[(iii)] the leading coefficient of $m_k$ is
equal to $1$ for every $k\in \mathbb{Z}_{+}$,
   \item[(iv)] $1$ is not a root of $m_k$ for
every $k\in \mathbb{Z}_{+}$,
   \item[(v)] $S_k(x)=\frac{m_k(x)}{(1-x)^{k+1}}$ for
every $k\in \mathbb{Z}_{+}$.
   \end{enumerate}
In particular, $S_k(x)\in \mathbb{Q}(x)$ for every $k
\in \mathbb{Z}_{+}$.
   \end{lemma}
   \begin{proof} We use induction on $k$. The
case of $k=0$ follows from \eqref{claan}. Suppose that
$m_0, \ldots, m_k \in \mathbb{Z}[x]$ have the required
properties for a fixed (unspecified) $k\in
\mathbb{Z}_+$. Note that
   \begin{align}\label{rety}
S_{k+1}(x)= (x(1^k+2^kx+3^kx^2+\ldots))^{\prime} =
(xS_k(x))^{\prime}, \quad x \in (0,1).
   \end{align}
By the induction hypothesis and (\ref{rety}), we have
   \begin{align}\label{ostatni}
S_{k+1}(x)=(xS_k(x))^{\prime}=\Big(\frac{xm_k(x)}{(1-x)^{k+1}}\Big)^{\prime}
=\frac{m_{k+1}(x)}{(1-x)^{k+2}}, \quad x \in (0,1),
   \end{align}
with
   \begin{equation} \label{mk+1}
m_{k+1}(x):=(xm_k^{\prime}(x)+m_k(x))(1-x)+(k+1)xm_k(x),
\quad x \in (0,1).
   \end{equation}
Now, it is easily seen that $m_{k+1} \in
\mathbb{Z}[x]$, the leading coefficients of $m_{k+1}$
and $m_k$ coincide and the degree of $m_{k+1}$ is
equal to $k$. Suppose that, contrary to our claim,
$x_0=1$ is a root of $m_{k+1}$. Then, by \eqref{mk+1},
we have
   \begin{equation*}
0=m_{k+1}(1)=(k+1)m_k(1),
   \end{equation*}
which contradicts $m_k(1)\neq 0$. Hence the fraction
appearing on the right-hand side of (\ref{ostatni}) is
irreducible. This also proves the uniqueness of
$m_{k+1}$.
   \end{proof}
We are now ready to construct a weighted shift on a
directed tree and a composition operator $C$ in an
$L^2$-space with the properties mentioned in
Introduction.
   \begin{theorem}\label{glowne}
Let $n$ be an integer greater than or equal to $2$.
Then there exists an injective non-quasinormal
weighted shift $S_\lambda\in \boldsymbol
B(\ell^2(V_{\infty}))$ on the directed tree
$\ttt_{\infty}$ which satisfies the condition
\eqref{qqq} with $A=S_\lambda$. Moreover, there exists
an injective non-quasinormal composition operator $C$
in $L^2$-space over a $\sigma$-finite measure space
satisfying the condition \eqref{qqq} with $A=C$.
   \end{theorem}
   \begin{proof}
By \cite[Lemma 4.3.1]{9}, every weighted shift on a
rootless and leafless directed tree with positive real
weights is unitarily equivalent to an injective
composition operator in an $L^2$-space over a
$\sigma$-finite measure space. Hence, it is enough to
construct a weighted shift on a rootless and leafless
directed tree with positive real weights which
satisfies the condition \eqref{qqq} with
$A=S_{\lambda}$. Let $\ttt_{\infty}$ be the directed
tree as in \eqref{drzewo} and \eqref{drzewo2}, and let
$S_{\lambda}$ be a weighted shift on $\ttt_{\infty}$
with weights as in \eqref{wagi1}, where the sequences
$\{\alpha_k\}_{k=1}^{\infty}$ and
$\{\beta_k\}_{k=1}^{\infty}$ are given by
   \begin{equation}\label{wagi}
\text{$\alpha_k=\sqrt{k^{n-1}q^{k-1}}$ and
$\beta_k=\sqrt{\frac{1}{k}c^{\frac{1}{n-1}}}$ for
$k\in \mathbb{N}$;}
   \end{equation}
here the numbers $q$ and $c$ are chosen to satisfy the
following three conditions
   \begin{gather} \label{qw0}
q \in \mathbb{Q} \cap (0,1) \text{ and } c
\in\mathbb{Q}\cap (0,\infty),
   \\ \label{qw} (S_{n-1}(q))^n = cS_{0}(q),
   \\ \label{qw2}
c^{\frac{k}{n-1}}\notin \mathbb{Q}, \quad
k\in\{1,2,\ldots,n-2\}.
   \end{gather}
(Note that the condition \eqref{qw2} is empty when
$n=2$.) Below, we will show how to construct such $c$
and $q$. Let $\{\gamma_k\}^{\infty}_{k=0}$ be a
sequence of positive real numbers uniquely determined
by the recurrence formulas (i) and (ii) of Proposition
\ref{js2}. Since, by \eqref{wagi}, the following
equalities hold
   \begin{align} \label{qw3}
\sum^{\infty}_{k=1} \alpha_k^2 & = S_{n-1}(q),
   \\ \label{qw4}
\sum^{\infty}_{k=1}\alpha_k^2 \beta_k^{2(i-1)} & =
c^{\frac{i-1}{n-1}} S_{n-i}(q), \quad i\in \mathbb Z,
   \end{align}
one can infer from \eqref{qw0} and Lemma \ref{mk} that
$\{\gamma_k\}^{\infty}_{k=0}$ is sequence of algebraic
numbers. According to Proposition \ref{js2}(i), for
every integer $i\Ge n-1$, the term $\gamma_i$ is a
geometric mean of $n-1$ preceding terms
$\gamma_{i-1}$, \ldots, $\gamma_{i-n+1}$. Hence, the
sequence $\{\gamma_i\}_{i=0}^{\infty}$ is bounded (see
the proof of Theorem \ref{mainX}). This combined with
\eqref{qw3}, \eqref{wagi} and Proposition \ref{ogrS}
shows that $S_\lambda\in \boldsymbol
B(\ell^2(V_{\infty}))$.

Now we prove that there exist $c$ and $q$ satisfying
\eqref{qw0}, \eqref{qw} and \eqref{qw2}. For this we
consider a new quantity $c_0\in (0,\infty)$ uniquely
determined by the equation $c=c_0(S_{n-1}(q))^{n-1}$.
Then the equality \eqref{qw} takes the form
$S_{n-1}(q)=c_0S_0(q)$, and thus by \eqref{claan} and
Lemma \ref{mk} we have
   \begin{equation}  \label{c0q}
c_0=\frac{S_{n-1}(q)}{S_0(q)}=\frac{m_{n-1}(q)}{(1-q)^{n-1}}.
   \end{equation}
It is easily seen that $c$ satisfies \eqref{qw0},
\eqref{qw} and \eqref{qw2} if and only if $c_0$ does.
Hence, it remains to construct $q$ and $c_0$. We may
assume that $n\Ge 3$. Let $p$ be a prime number and
let $m_{n-1}(x) = a_{n-2} x^{n-2} + \ldots + a_0$ with
$a_0, \ldots, a_{n-2} \in \mathbb Z$. Set
$q=\frac{1}{p}$. Then, by \eqref{c0q}, we have
   \begin{equation*}
c_0 =
\frac{p(a_0p^{n-2}+\ldots+a_{n-2})}{(p-1)^{n-1}}.
   \end{equation*}
Since, by Lemma \ref{mk}(iii), $a_{n-2}=1$ and $p$ is
prime, an elementary reasoning shows that the number
$c_0$ satisfies \eqref{qw2}. This gives the required
$q$ and $c$.

To complete the proof, it suffices to consider tree
disjunctive cases.
   \begin{case}
$(S^{*}_{\lambda}S_{\lambda})^p \neq S^{*p}_{\lambda}
S^{p}_{\lambda}$ for every $ p \in
\{2,3,\ldots,n-1\}$.
   \end{case}
Indeed, otherwise
$(S^{*}_{\lambda}S_{\lambda})^p=S^{*p}_{\lambda}S^{p}_{\lambda}$
for some $ p \in \{2,3,\ldots,n-1\}$. In view of
Proposition \ref{js2}(iii), we have
   \begin{equation*}
\Big(\sum^{\infty}_{k=1} \alpha_k^2
\Big)^p=\sum^{\infty}_{k=1}\alpha_k^2
\beta_k^{2(p-1)},
   \end{equation*}
which together with \eqref{qw3} and \eqref{qw4} yields
   \begin{equation} \label{qw5.5}
S_{n-1}(q)^p=c^{\frac{p-1}{n-1}}S_{n-p}(q).
   \end{equation}
It follows from \eqref{qw0} and Lemma \ref{mk} that
$S_{n-1}(q) \in \mathbb{Q}$ and $S_{n-p}(q)\in
\mathbb{Q}$. This and \eqref{qw5.5} lead to a
contradiction with \eqref{qw2}.
   \begin{case} $(S^{*}_{\lambda}S_{\lambda})^p \neq S^{*p}_{\lambda}
S^{p}_{\lambda}$ for $p=n+1$.
   \end{case}
Indeed, otherwise arguing as in Case 1 we show that
   \begin{equation}  \label{qw6}
S_{n-1}(q)^{n+1}=c^\frac{n}{n-1}S_{-1}(q).
   \end{equation}
Since the numbers $S_{n-1}(q)$ and $c^{\frac{n}{n-1}}$
are algebraic, the equality \eqref{qw6} leads to a
contradiction with the fact that the number
$S_{-1}(q)=-\frac{\ln(1-q)}{q}$ is transcendental (see
\eqref{claan} and Corollary \ref{logarytm}).
   \begin{case} $(S^{*}_{\lambda} S_{\lambda})^p \neq S^{*p}_{\lambda}
S^{p}_{\lambda}$ for every $p \in \{n+2,n+3,\ldots\}$.
   \end{case}
Indeed, otherwise $(S^{*}_{\lambda} S_{\lambda})^p =
S^{*p}_{\lambda} S^{p}_{\lambda}$ for some $ p \in
\{n+2,n+3,\ldots\}$. Hence, by Proposition
\ref{js2}(ii) and \eqref{qw4}, we have
   \begin{align}   \label{qw7}
\gamma_{p-i-1}^{2p} =
\gamma_{p-i-1}^2\cdots\gamma_{0}^2
c^\frac{i-1}{n-1}S_{n-i}(q), \quad i\in
\{1,2,\ldots,{p-1}\}.
   \end{align}
Substituting $i=n+1$ into \eqref{qw7}, we get
   \begin{align*}
\gamma_{p-n-2}^{2p} =
\gamma_{p-n-2}^2\cdots\gamma_{0}^2
c^\frac{n}{n-1}S_{-1}(q).
   \end{align*}
Since $\{\gamma_k\}^{\infty}_{k=0}$ is sequence of
algebraic numbers, we can argue as in Case 2 to get a
contradiction with the transcendentality of
$S_{-1}(q)$. This completes the proof.
   \end{proof}

\textbf{Acknowledgements.} I would like to thank my supervisor, prof. Jan Stochel for encouragement and motivation, as well as substantial help he provided me while working on this paper.
   \bibliographystyle{amsalpha}
   
\end{document}